\documentclass[11pt]{amsart}
\usepackage{xypic}
\usepackage{amsmath}
\usepackage{amssymb}
\usepackage{amsthm}
\usepackage{hyperref}
\usepackage{graphicx}
\usepackage{tikz}

\newcommand{\cc}{\mathbb{C}}
\renewcommand{\sl}{\mathrm{sl}}

\newtheoremstyle{thm}{15 pt}{10 pt}{\itshape}{}{\bfseries}{.}{.5em}{}
\newtheorem{theorem}{Theorem}
\numberwithin{theorem}{section}
\newtheorem{corollary}[theorem]{Corollary}
\newtheorem{question}[theorem]{Question}
\newtheoremstyle{rem}{15 pt}{10 pt}{\normalfont}{}{\bfseries}{.}{.5em}{}
\theoremstyle{rem}
\newtheorem{remark}[theorem]{Remark}

\begin{document}

\thispagestyle{empty}
\title[Quasipositive minimal braids]{Minimal braid representatives of quasipositive links}
\author[K. Hayden]{Kyle Hayden} \address{Boston College, Chestnut Hill, MA 02467} \email{kyle.hayden@bc.edu}

\maketitle

\begin{abstract}
We show that every quasipositive link has a quasipositive minimal braid representative, partially resolving a question posed by Orevkov.  These quasipositive minimal braids are used to show that the maximal self-linking number of a quasipositive link is bounded below by the negative of the minimal braid index, with equality if and only if the link is an unlink. This implies that the only amphicheiral quasipositive links are the unlinks, answering a question of Rudolph's. 
\end{abstract}

\section{Introduction}

Quasipositive links in $S^3$ were introduced by Rudolph \cite{rudolph:qp-alg} and defined in terms of special braid diagrams, the details of which we recall below. These links possess a variety of noteworthy features. Perhaps most strikingly, results of Rudolph \cite{rudolph:qp-alg} and Boileau-Orevkov \cite{bo:qp} show that quasipositive links are precisely those links which arise as transverse intersections of the unit sphere $S^3 \subset \cc^2$ with complex plane curves $f^{-1}(0)\subset \cc^2$, where $f$ is a non-constant polynomial. The hierarchy of braid positive, positive, strongly quasipositive, and quasipositive links interacts in compelling ways with conditions such as fiberedness \cite{E-VHM:fibered,hedden:pos}, sliceness \cite{rudolph:qp-obstruction}, homogeneity \cite{baader:homogeneity}, and symplectic/Lagrangian fillability \cite{bo:qp,positivity}, and quasipositive links have well-understood behavior with respect to invariants such as the four-ball genus, the maximal self-linking number, and the Ozsv\'ath-Szab\'o concordance invariant $\tau$ \cite{hedden:pos}. For a different perspective, we can view quasipositive braids as a monoid in the mapping class group of a disk with marked points, where they lie inside the contact-geometrically important monoid of right-veering diffeomorphisms; see \cite{E-VHM:monoids} for more details.

The braid-theoretic description of quasipositivity is as follows: A braid is called \emph{quasipositive} if it is the closure of a word 
$$\prod_i \omega_i \sigma_{j_i} \omega_i^{-1},$$
where $\omega_i$ is any word in the braid group and $\sigma_{j_i}$ is a positive standard generator. A link is then called \emph{quasipositive} if it has a quasipositive braid representative. However, an arbitrary braid representative of a quasipositive link need not be a quasipositive braid. Along these lines, Orevkov \cite{orevkov:markov} posed the following question:

\begin{question}\label{ques:orevkov}  Let $\mathcal{L}$ be a quasipositive link and $\beta$ a minimal braid index representative of $\mathcal{L}$. Is  $\beta$ is quasipositive?
\end{question}

Partial resolutions to this question have appeared in \cite{E-VHM:fibered} and \cite{fel-krc:twists}. In particular, Etnyre and Van Horn-Morris showed in \cite{E-VHM:fibered} that the answer to Question~\ref{ques:orevkov} is \emph{yes} for fibered strongly quasipositive links. (For contrast, we point out that the answer to the analogue of Question~\ref{ques:orevkov} for positive braids is \emph{no}, as Stoimenow has provided examples of braid positive knots that have no positive minimal braid representative in \cite{stoimenow:crossing}. See also \cite[\S1]{stoimenow:3-braids}.) The main purpose of this note is to provide another partial answer to Question~\ref{ques:orevkov}.

\begin{theorem} \label{thm:main}
Every quasipositive link has a quasipositive minimal braid index representative.
\end{theorem}

This claim follows quickly from some recent results in the theory of braid foliations, namely LaFountain and Menasco's proof of the Generalized Jones Conjecture in \cite{lf-m:jones}. Our method of proof is similar to that of Etnyre and Van Horn-Morris in \cite{E-VHM:fibered};  their later paper \cite{E-VHM:monoids} offers an updated perspective.

A few simple consequences follow from Theorem~\ref{thm:main}. First, by considering the self-linking number of a quasipositive minimal braid index representative of a quasipositive link, we obtain a lower bound on the maximal self-linking number $\overline{\sl}$ in terms of the minimal braid index $b$:

\begin{theorem}\label{thm:sl}
If $\mathcal{L}$ is a quasipositive link, then $$\overline{\sl}(\mathcal{L}) \geq - b(\mathcal{L}),$$
with equality if and only $\mathcal{L}$ is an unlink.
\end{theorem}

The calculation underlying Theorem~\ref{thm:sl} also lets us resolve an earlier question of Lee Rudolph's from \cite[Problem~9.2]{morton:problems}:

\begin{question}\label{ques:rudolph}
Are there any amphicheiral quasipositive links other than the unlinks?
\end{question}

At the time this question was asked, it was already known that nontrivial strongly quasipositive knots were chiral; see \cite[Remark 4]{rudolph:positive} for a discussion of precedent results. Additional evidence for a negative answer came in the form of strong constraints on invariants of amphicheiral quasipositive links (including their being slice \cite{wu:deformations}).  We confirm that the answer to Rudolph's question is \emph{no}.

\begin{corollary}\label{cor:chiral}
If a link $\mathcal{L}$ and its mirror $m(\mathcal{L})$ are both quasipositive, then $\mathcal{L}$ is an unlink. In particular, an amphicheiral quasipositive link is an unlink.
\end{corollary}

After recalling the necessary background in Section~\ref{sec:background}, we supply proofs for the above results in Section~\ref{sec:minimal}.

\subsection*{Acknowledgements.} The author thanks John Baldwin, Peter Feller, and Eli Grigsby for several stimulating conversations and for introducing him to LaFountain and Menasco's proof of the Generalized Jones Conjecture.

\section{Background}\label{sec:background}

The Generalized Jones Conjecture, first confirmed by Dynnikov-Prasolov \cite{d-p:jones}, relates the writhe $w$ and braid index $n$ of braids within a given link type.

\begin{theorem}[Generalized Jones Conjecture, \cite{d-p:jones}]
Let $\beta$ and $\beta_0$ be closed braids with the same link type $\mathcal{L}$, where $n(\beta_0)$ is minimal for $\mathcal{L}$. Then there is an inequality $$|w(\beta)-w(\beta_0)| \leq n(\beta)-n(\beta_0).$$ \end{theorem}

Recall Bennequin's formula for the self-linking number of a braid $\beta$: $$\sl(\beta)=w(\beta)-n(\beta).$$  It follows from the Generalized Jones Conjecture, Bennequin's formula, and the transverse Alexander theorem  that a minimal braid index representative of $\mathcal{L}$ achieves the maximal self-linking number among all transverse representatives of $\mathcal{L}$, denoted $\overline{\sl}(\mathcal{L})$. 
For any braid $\beta$ representing a link type $\mathcal{L}$, we can plot the pair $(w(\beta),n(\beta))$ in a plane. The \emph{cone} of $\beta$ is the collection of all pairs $(w,n)$ realized by braids which are stabilizations of $\beta$; see Figure~\ref{fig:cone_background} for an example. If $\beta_0$ is a minimal braid index representative of $\mathcal{L}$, we see that the right edge of its cone consists of all pairs $(w,n)$ corresponding to braids achieving the maximal self-linking number of $\mathcal{L}$.

\begin{figure}[h!] \centering \def\svgwidth{200pt} 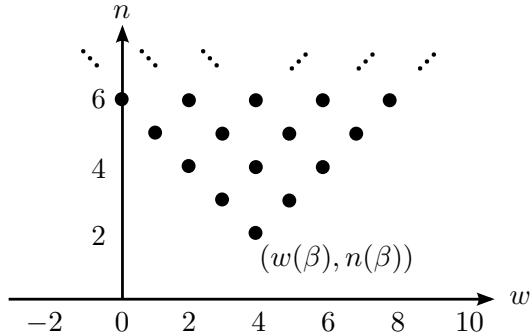 \caption{The cone of a braid $\beta$ with $(w(\beta),n(\beta))=(4,2)$.}\label{fig:cone_background} \end{figure}

The other tool central to the proof of Theorem~\ref{thm:main} is due to Orevkov and concerns braid moves that preserve quasipositivity.

\begin{theorem}[\cite{orevkov:markov}]\label{thm:stab}
Suppose the braids $\beta$ and $\beta'$ are related by positive (de)stabilization. Then $\beta$ is quasipositive if and only if $\beta'$ is quasipositive.
\end{theorem}

\begin{remark} In \cite{orevkov:markov}, Orevkov views an $n$-stranded braid as an isotopy class of $n$-valued functions $f: [0,1] \to \cc$ such that each of $f(0)$ and $f(1)$ equals $\{1,2,\ldots,n\}\subset \cc$.  A braid is then quasipositive if its equivalence class contains a representative that can be expressed as a product of conjugates of the standard generators. For us, it is more convenient to consider \emph{closed} braids (up to isotopy through closed braids). Two closed braids are equivalent if and only if they can be expressed as closures of conjugate open braids. Since quasipositivity is a property of conjugacy classes of open braids,  Theorem~\ref{thm:stab} holds equally well for closed braids.
 \end{remark}

\section{Quasipositive minimal braids}\label{sec:minimal}

We proceed to the proof of the of the main result, namely that every quasipositive link has a quasipositive minimal braid representative.

\begin{proof}[Proof of Theorem~\ref{thm:main}]
Let $\mathcal{L}$ be a quasipositive link with a minimal braid index representative $\beta_0$ and a quasipositive braid representative $\beta_+$.  Since the slice-Bennequin inequality is sharp for quasipositive links (\cite{rudolph:qp-obstruction}, c.f. \cite{hedden:pos}), $\beta_+$ achieves the maximal self-linking number for $\mathcal{L}$. As noted above, it follows that $(w(\beta_+),n(\beta_+))$ lies along the right edge of the cone of $\beta_0$. The braids $\beta_0$ and $\beta_+$ have the same link type, so \cite[Proposition~1.1]{lf-m:jones} implies that there are braids $\beta_0'$ and $\beta_+'$ obtained from $\beta_0$ and $\beta_+$ by negative and positive stabilization, respectively, such that $\beta_0'$ and $\beta_+'$ cobound embedded annuli. Note that $\beta_0'$ and $\beta_+'$ lie along the left and right edges of the cone, respectively, as depicted on the left side of Figure~\ref{fig:cone_main}. We also note that $\beta_+'$ is quasipositive since it is obtained from $\beta_+$ by positive stabilization. 

Next, as in (the proof of) \cite[Proposition~3.2]{lf-m:jones}, we can find braids $\beta_0''$ and $\beta_+''$  obtained from $\beta_0'$ and $\beta_+'$ by braid isotopy, destabilization, and exchange moves such that $w(\beta_+'')=w(\beta_0'')$ and $n(\beta_+'')=n(\beta_0'')$. 
We claim that $\beta_+''$ has minimal braid index (as does $\beta_0''$). Indeed, since $\beta_0'$ and $\beta_+'$ lie on the left and right edges of the cone of $\beta_0$, the destabilizations applied to them must be negative and positive, respectively. Given this and the fact that exchange moves preserve writhe and braid index, we see that $\beta_0''$ and $\beta_+''$ must also lie on the left and right edges of the cone of $\beta_0$, respectively. But since these braids occupy the same $(w,n)$-point, they must lie where the edges of the cone meet. As depicted on the right side of Figure~\ref{fig:cone_main}, this implies that $\beta_0''$ and $\beta_+''$ have minimal braid index.

Finally, we show that $\beta_+''$ is quasipositive. As noted above, any destabilizations of $\beta_+'$ must be positive, and these preserve quasipositivity by Theorem~\ref{thm:stab}. An exchange move also preserves quasipositivity, since it can be expressed as a combination of one positive stabilization, one positive destabilization, and a number of conjugations; see \cite[Figure~8]{birman-wrinkle} for a proof.
\end{proof}

\begin{figure}[t] \centering \def\svgwidth{\linewidth} 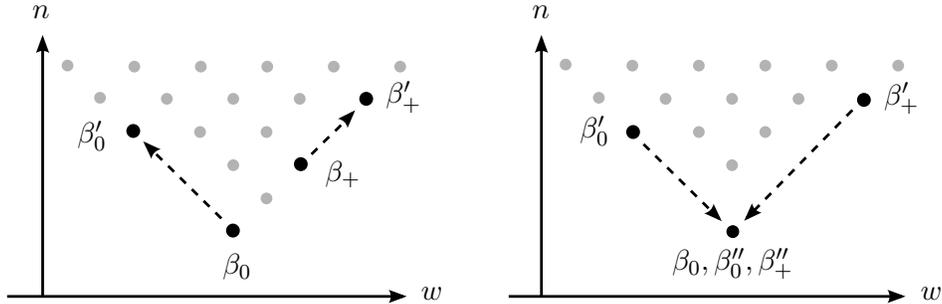 \caption{On the left, $\beta_0'$ and $\beta_+'$ are obtained from $\beta_0$ and $\beta_+$ by negative and positive stabilization, respectively. Then, on the right, $\beta_0''$ and $\beta_+''$ are obtained from $\beta_0'$ and $\beta_+'$ by negative and positive destabilization, respectively.} \label{fig:cone_main}\end{figure}

\begin{remark} 
The question of whether or not \emph{all} minimal braid index representatives of a quasipositive link are quasipositive remains open.  The answer is  seen to be \emph{yes} for transversely simple link types: Beginning with a quasipositive braid representative, the transverse Markov theorem implies that any minimal braid index representative can be related to it by positive (de)stabilization, which preserves quasipositivity. By the same reasoning, the answer to Question~\ref{ques:orevkov} is \emph{yes} for any link type that has a unique transverse class achieving its maximal self-linking number (but is not necessarily transversely simple).  This is the case for fibered strongly quasipositive links, as shown by Etnyre and Van Horn-Morris. But it fails to hold even for non-fibered strongly quasipositive links; as pointed out by Etnyre and Van Horn-Morris, there are infinite families of 3-braids found by Birman and Menasco in \cite{b-m:simplicity} which are (strongly) quasipositive and of minimal braid index but not transversely isotopic.
\end{remark}

\begin{remark}
As pointed out by Eli Grigsby, the proof of Theorem~\ref{thm:main} can be mirrored to show that any property of closed braids that is \begin{enumerate}
\item preserved under transverse isotopy and
\item satisfied by at least one braid representative of $\mathcal{L}$ with maximal self-linking number
\end{enumerate}
is also satisfied by at least one minimal braid index representative of $\mathcal{L}$.\end{remark}

Now we obtain the lower bound in Theorem~\ref{thm:sl} by applying Bennequin's formula to a quasipositive minimal braid.

\begin{proof}[Proof of Theorem~\ref{thm:sl}] 
Recall that a quasipositive braid always achieves the maximal self-linking number and that the writhe of a quasipositive braid is always greater than or equal to zero, with equality if and only if the braid is trivial. Applying this to a quasipositive minimal braid index representative $\beta$ of $\mathcal{L}$, we have
\begin{align*}
\overline{\sl}(\mathcal{L}) &= \sl(\beta) \\
&= w(\beta)-n(\beta) \\
& \geq -b(\mathcal{L}),
\end{align*}
where equality holds if and only if the writhe of $\beta$ is zero,  i.e.\ if and only if $\beta$ is the trivial braid.
\end{proof}

Finally, we prove the corollary that resolves Question~\ref{ques:rudolph}.

\begin{proof}[Proof of Corollary~\ref{cor:chiral}]
If $\beta$ is a minimal braid index representative of a link $\mathcal{L}$, then the mirror $m(\beta)$ is minimal for $m(\mathcal{L})$.  As in the proof of Theorem~\ref{thm:sl}, $\mathcal{L}$ and $m(\mathcal{L})$ being quasipositive implies that $w(\beta)\geq0$ and $-w(\beta)=w(m(\beta))\geq 0$. These together imply $w(\beta)=0$, and thus $\mathcal{L}$ is an unlink.\end{proof}

\bibliographystyle{abbrv}
\bibliography{biblio}

\end{document}